\documentclass[12pt]{amsart}
\usepackage{amsmath, amssymb, amsthm, a4wide}
\usepackage{hyperref}
\usepackage{caption}
\usepackage[usenames]{color}
\usepackage{tikz}

\newtheorem{theorem}{Theorem}[section]
\newtheorem{lemma}[theorem]{Lemma}
\newtheorem{claim}[theorem]{Claim}
\newtheorem{rem}[theorem]{Remark}
\newtheorem{definition}[theorem]{Definition}
\newtheorem{proposition}[theorem]{Proposition}
\newtheorem{obs}[theorem]{Observation}

\newtheorem{conj}[theorem]{Conjecture}
\newtheorem*{notat*}{Notation}

\newcommand{\AT}[2]{\mathcal{AT}_{#1}(#2)}
\newcommand{\Pairs}{{\Delta}}

\begin{document}
\title{The maximal number of $3$-term arithmetic progressions in finite sets in different geometries}

\renewcommand{\shorttitle}{Maximal number of $3$-term arithmetic progressions in different geometries}

\author{Itai Benjamini}
\address{Department of Mathematics, Weizmann Institute of Science, Rehovot 7610001, Israel}

\author{Shoni Gilboa}
\address{Department of Mathematics and Computer Science, The Open University of Israel, Raanana 4353701, Israel}

\begin{abstract}
Green and Sisask showed that the maximal number of $3$-term arithmetic progressions in $n$-element sets of integers is $\lceil n^2/2\rceil$; it is easy to see that the same holds if the set of integers is replaced by the real line or by any Euclidean space.
We study this problem in general metric spaces, where a triple $(a,b,c)$ of points in a metric space is considered a \emph{$3$-term arithmetic progression} if $d(a,b)=d(b,c)=\frac{1}{2}d(a,c)$.
In particular, we show that the result of Green and Sisask extends to any Cartan--Hadamard manifold (in particular, to the hyperbolic spaces), but does not hold in spherical geometry or in the $r$-regular tree, for any $r\geq 3$.
\end{abstract}


\maketitle

\section{Introduction}\label{sec:intro}

It was shown in \cite[Theorem 1.2]{GS} that the maximal number of $3$-term arithmetic progressions in $n$-element sets of integers is $\lceil n^2/2\rceil$ 
(counting increasing, decreasing and constant progressions).
The maximum is attained for $n$-term arithmetic progressions, but also for other sets (completely characterized in \cite[Theorem 1.2]{GS}).

Combined with some tools from additive combinatorics, this result was used in \cite{GS} to obtain their main result that $\lceil n^2/2\rceil$ is also the maximal number of $3$-term arithmetic progressions in $n$-element subsets of the additive group ${\mathbb Z}/ p{\mathbb Z}$ for prime $p$, provided that $n/p$ is smaller than some absolute constant.

Additive structure is probably the most natural context of arithmetic progressions, but it may be viewed also as a metric notion, which is the direction we pursue here.

\begin{definition}Let $M$ be a metric space. 

We say that $(a,b,c)\in M^3$ is a \emph{$3$-term arithmetic progression} in $M$ if 
$$d_M(a,b)=d_M(b,c)=\textstyle\frac{1}{2}d_M(a,c),$$
where $d_M$ is the metric of $M$.
For any set $A\subseteq M$, let 
$$\AT M A:=\left\{(a,b,c)\in A^3\mid d_M(a,b)=d_M(b,c)=\textstyle\frac{1}{2}d_M(a,c)\right\}$$
be the set of $3$-term arithmetic progressions in the set $A$. 
For every positive integer $n$, let 
$$\mu_n(M):=\max\left\{|\AT M A|: A\subseteq M,\, |A|=n\right\}$$
be the maximal number of $3$-term arithmetic progressions in $n$-element subsets of $M$.
\end{definition}

As already mentioned, it was shown in \cite{GS} that $\mu_n({\mathbb Z})=\lceil n^2/2\rceil$ for every $n$, and the same argument shows that for every $n$,
\begin{equation}\label{eq:R}
\mu_n({\mathbb R})=\lceil n^2/2\rceil.
\end{equation} 
This yields, by a simple projection argument, that for any $k$, with respect to the Euclidean metric, $\mu_n({\mathbb R}^k)=\lceil n^2/2\rceil$ for every $n$.

We show that this extends to a rather large class of metric spaces. 
First, let us recall some basic notions.
Let $M$ be a metric space;
a curve $\gamma:I\to M$, where $I$ is a connected subset of the real line, is  a \emph{geodesic} if $d_M(\gamma(y),\gamma(x))=y-x$ for every $x<y$ in $I$;  
a set $\Gamma\subseteq M$ is a \emph{geodesic segment with endpoints $p,q$} if there is a geodesic $\gamma:[a,b]\to M$ such that $\Gamma=\gamma([a,b])$, $p=\gamma(a)$ and $q=\gamma(b)$; 
the metric space $M$ is \emph{uniquely geodesic} if any two distinct points in $M$ are the endpoints of a unique geodesic segment;
finally, a curve $\gamma:I\to M$, where $I$ is a connected subset of the real line, is  a \emph{local geodesic} if around every $a\in I$ there is an interval $I_a$ such that the restriction of $\gamma$ to $I\cap I_a$ is a geodesic.  

\begin{theorem}\label{thm:hyper}
Let $M$ be a uniquely geodesic Riemannian manifold in which every local geodesic is a geodesic.
Then, for every $n$,
$$\mu_n(M)=\lceil n^2/2\rceil.$$
Moreover, any set $A$ of $n$ points in $ M$ for which $|\AT M A|=\lceil n^2/2\rceil$ is contained in the image of a geodesic.
\end{theorem}

In particular, Theorem \ref{thm:hyper} applies to the hyperbolic spaces, and more generally, to any \emph{Cartan--Hadamard manifold}, i.e., complete simply connected Riemannian manifold that has everywhere nonpositive sectional curvature (see, e.g., \cite{BGS, CE}).
However, the result does not extend to the wider class of metric spaces of global nonpositive curvature, in the sense of A. D. Alexandrov
(also coined CAT($0$) spaces by Gromov, in honor of Cartan, Alexandrov and Toponogov; note that each such metric space is uniquely geodesic, and every local geodesic in it is a geodesic; see, e.g., \cite{BH, BBI}).
For instance,  let $\mathbb T_r$ be the (discrete) $r$-regular tree, $r\geq 2$, equipped with the graph metric, and let $\hat{\mathbb T}_r$ be the corresponding metric graph, where all the edges have unit length. 
The metric tree $\hat{\mathbb T}_r$ is a \emph{Hadamard space}, i.e., a complete globally nonpositively curved metric space.
We show (see Proposition \ref{prop:trees}) that
\begin{equation}\label{eq:trees}
\limsup_{n\to\infty}\frac{\mu_n({\mathbb T}_r)}{n^2}\geq\frac{1}{2}+\frac{(r-2)^2}{2r^2}.
\end{equation}
Hence, for every $r\geq 3$, since obviously $\mu_n({\hat{\mathbb T}_r})\geq\mu_n({\mathbb T_r})$ for every $n$, it holds that
$$\limsup_{n\to\infty}\frac{\mu_n(\hat{\mathbb T}_r)}{n^2}\geq\limsup_{n\to\infty}\frac{\mu_n({\mathbb T}_r)}{n^2}>\frac{1}{2}=\limsup_{n\to\infty}\frac{\lceil n^2/2\rceil}{n^2}.$$
Note that $\mu_n(\hat{\mathbb T}_r)\leq n^2-2n+2$ for every $n$, by the following simple claim.
\begin{claim}\label{obs:unigeodesic}
Let $M$ be a uniquely geodesic metric space (with more than one point). Then, for every $n$,
$$\lceil n^2/2\rceil\leq \mu_n(M)\leq n^2-2n+2.$$
Moreover, there is a uniquely geodesic metric space $M_0$ such that $\mu_n(M_0)=n^2-2n+2$ for every $n$.
\end{claim}

Next, we consider some positively curved metric spaces.
\begin{theorem}\label{thm:S1}
For every $n\neq 2$,
$$\mu_n(S^1)= \frac{1}{2}n^2+\begin{cases}
n & n\bmod 4=0,\\
\frac{1}{2}n& n\bmod 4=1,\\
2 & n\bmod 4=2,\\
\frac{1}{2}n-1 & n\bmod 4=3.
\end{cases}$$
\end{theorem}

In particular, \eqref{eq:R} and Theorem \ref{thm:S1} imply that $\mu_n(S^1)>\lceil n^2/2\rceil$ for every $n\geq 4$.

We do not know, for general $n$, what is the maximal number of $3$-term arithmetic progressions in $n$-element subsets of the $2$-dimensional sphere $S^2$. However, we were able to show (see Proposition \ref{prop:equator} and the discussion thereafter) that
\begin{equation}\label{eq:S2S1}
\mu_n(S^2)>\mu_n(S^1) \text{ for every } n\geq 5.
\end{equation}

\smallskip
For the $\ell$-dimensional lattice graph ${\mathbb Z}^{\ell}$ (where two vertices are adjacent if the Euclidean distance between them is $1$), with respect to the graph metric, we show (see Proposition \ref{prop:lattices}) that 
\begin{equation}\label{eq:lattices}
\mu_n\left({\mathbb Z}^{\ell}\right)=\Omega\left(n^{3-\frac{1}{\ell}}\right).
\end{equation}

\medskip
Finally, we show (see Proposition \ref{prop:goodman}) that
\begin{equation}\label{eq:goodman}
\max_{M \text{ is a metric space}}\mu_n(M)=\frac{1}{4}n^3-\frac{1}{2}n^2+\Theta(n).
\end{equation} 

\smallskip
This work was inspired by the analogous route taken by the study of the isoperimetric problem, from the classical setting in the Euclidean plane to diverse geometric context (see, e.g., \cite{Oss, BZ, Cha} and the references within), including spherical geometry (see, e.g., \cite{Gro}), Cartan--Hadamard manifolds (see, e.g., \cite{KK} and the references within), and graphs (see, e.g., \cite{Lea} and the references within).

\medskip
The rest of the paper is organized as follows. In Section \ref{sec:prelim} we introduce some notation and make some preliminary observations.
In Section \ref{sec:hyper} we prove Theorem \ref{thm:hyper} and Claim \ref{obs:unigeodesic}.
In Section \ref{sec:spheric} we prove Theorem \ref{thm:S1}, confirm \eqref{eq:S2S1} and make some additional observations regarding 
the problem in the $2$-dimensional sphere $S^2$.
In Section \ref{sec:misc} we prove \eqref{eq:trees},\eqref{eq:lattices} and \eqref{eq:goodman}.
Finally, in Section \ref{sec:open}, we suggest several directions for future study.

\subsection*{Acknowledgements} 
We thank Lev Buhovski, Dan Hefetz, Bo'az Klartag and Pierre Pansu for fruitful discussions, and the anonymous referees for their helpful comments. 

\section{Notation and Preliminaries}\label{sec:prelim}

For a point $b$ in a metric space $M$ and a nonnegative real number $d$, let
$$S_M(b; d):=\{x\in M\mid d_M(b,x)=d\}, \quad B_M(b; d):=\{x\in M\mid d_M(b,x)\leq d\}.$$

For a finite set $A$ of points in a metric space $ M$ and $b\in A$, denote
$$w_A(b):=|\{(x,y,z)\in\AT M A: y=b\}|.$$
Clearly,
\begin{equation}\label{eq:sum_w}
|\AT M A|=\sum_{b\in A}w_A(b).
\end{equation}

\begin{obs}\label{obs:parity}
Note that $(b,b,b)\in\AT M A$ for every $b\in A$ and that if $(a,b,c)\in\AT M A$, then $(c,b,a)\in\AT M A$ as well. 
Therefore, $|\AT M A|-|A|$ is always even. Moreover, $w_A(b)$ is odd for every $b\in A$.
\end{obs}

\begin{claim}\label{obs:unique}
Let $M$ be a metric space such that for any two points $a,c\in M$ there is at most one point $b\in M$ such that $(a,b,c)$ is a $3$-term arithmetic progression. 
Then, for every $n$,
$$\mu_n(M)\leq n^2-2n+2.$$
\end{claim}

\begin{proof}
Let $A$ be a set of $n$ points in $M$, and consider the graph $G=(A,E)$, where 
$$E:=\left\{\{a,c\}\subseteq A\mid \text{there is no }b\in A \text{ such that } (a,b,c)\in\AT M A\right\}.$$
If the graph $G$ is not connected, then let $a,c$ be vertices in different connected components of $G$ for which $d_M(a,c)$ is minimal. Obviously, the vertices $a,c$ are not adjacent, i.e., there is $b\in A$ such that $(a,b,c)$ is a $3$-term arithmetic progression in $M$. At least one of the vertices $a,c$ is not in the same connected component as $b$, contradicting the minimality of $d_M(a,c)$, since $d_M(a,b)=d_M(b,c)=\textstyle\frac{1}{2}d_M(a,c)<d_M(a,c)$.
Thus, the graph $G$ is necessarily connected. In particular, $|E|\geq n-1$ and hence, 
\begin{equation*}
|\AT M A|\leq n^2-2|E|\leq n^2-2(n-1)=n^2-2n+2.
\qedhere\end{equation*}
\end{proof}
We remark that there are metric spaces satisfying the assumption of Claim \ref{obs:unique} for which the bound given in the claim is tight for every $n$. One such example is the metric space $M_0$ presented in the proof of Claim \ref{obs:unigeodesic}. A similar, simpler, example is a bipartite graph with only one vertex in one part and infinitely many vertices in the other part, equipped with the graph metric.

\section{The problem in Cartan--Hadamard manifolds (and beyond)}\label{sec:hyper}

\begin{lemma}\label{lem:geodesic}
Let $M$ be a uniquely geodesic metric space, and let $(a,b,c)$ be a nonconstant $3$-term arithmetic progression in $M$. 
Then, $b$ lies on the geodesic segment with endpoints $a,c$.
\end{lemma}

\begin{proof}
For simplicity, denote $\delta:=d_M(a,b)=d_M(b,c)=\textstyle\frac{1}{2}d_M(a,c)>0$.
There are geodesics $\gamma_a:[-\delta,0]\to M, \,\,\, \gamma_c:[0,\delta]\to M$ such that 
\begin{equation*}
\gamma_a(-\delta)=a,\quad \gamma_a(0)=\gamma_c(0)=b,\quad \gamma_c(\delta)=c.
\end{equation*}
Define $\gamma:[-\delta,\delta]\to M$ by 
\begin{equation*}
\gamma(t):=\begin{cases}
\gamma_a(t) & -\delta\leq t\leq 0,\\
\gamma_c(t) & 0\leq t\leq\delta. 
\end{cases}
\end{equation*}
Then, for every $-\delta\leq x< 0< y\leq \delta$,
\begin{align*}
d_M(\gamma(y),\gamma(x))&\leq d_M(\gamma(y),\gamma(0))+d_M(\gamma(0),\gamma(x))\\
&=d_M(\gamma_c(y),\gamma_c(0))+d_M(\gamma_a(0),\gamma_a(x))=(y-0)+(0-x)=y-x,\\
d_M(\gamma(y),\gamma(x))&\geq d_M(\gamma(\delta),\gamma(-\delta))-d_M(\gamma(\delta),\gamma(y))-d_M(\gamma(x),\gamma(-\delta))\\
&= d_M(c,a)-d_M(\gamma_c(\delta),\gamma_c(y))-d_M(\gamma_a(x),\gamma_a(-\delta))\\
&=2\delta-(\delta-y)-(x-(-\delta))=y-x
\end{align*}
and hence $d_M(\gamma(y),\gamma(x))=y-x$. Obviously, this is also true if $-\delta\leq x< y\leq 0$ or $0\leq x< y\leq \delta$. Therefore, $\gamma$ is a geodesic, and the result follows.
\end{proof}

We may now prove Claim \ref{obs:unigeodesic}.
\begin{proof}[Proof of Claim \ref{obs:unigeodesic}]
Take an arbitrary nonconstant geodesic $\gamma:I\to M$. For every $n$ we may find an $n$-term arithmetic progression $A_n\subset I$ and then,
\begin{equation*}
\mu_n(M)\geq|\AT M {\gamma(A_n)}|=|\AT {\mathbb R} {A_n}|=\lceil n^2/2\rceil,
\end{equation*}
proving the lower bound.
The upper bound follows immediately from Lemma \ref{lem:geodesic} and Claim \ref{obs:unique}.

Finally, consider the metric space $M_0$ obtained by endowing the complex plane with the metric 
$$d_{M_0}(z,w)=\begin{cases}
|z-w| & \text{either } w\neq 0 \text{ and } z/w\in{\mathbb R} \text{ or } w=0,\\
|z|+|w| & \text{otherwise}.
\end{cases}$$
The metric space $M_0$ is clearly uniquely geodesic. For every $n$, let $C_{n-1}$ be a set of $n-1$ complex points with modulus $1$. 
Then, 
$$
\AT {M_0} {C_{n-1}\cup\{0\}}=\{(a,0,c)\mid a,c\in C_{n-1},\, a\neq c\}\cup\{(b,b,b)\mid b\in C_{n-1}\cup\{0\}\}
$$
and hence
$$|\AT {M_0} {C_{n-1}\cup\{0\}}|=(n-1)(n-2)+n=n^2-2n+2.$$
Therefore, $\mu_n(M_0)=n^2-2n+2$.
\end{proof}

Theorem \ref{thm:hyper} will easily follow, by using Lemma \ref{lem:geodesic}, from the following proposition.

\begin{proposition}\label{prop:GSgen}
Let $M$ be a metric space and let $\mathcal L$ be a family of subsets of $M$, that we will refer to as `lines', such that the following conditions hold:
\begin{itemize}
\item Each line in $\mathcal L$  is isometric to a subset of the real line. 
\item For any two distinct points in $M$ there is a unique line in $\mathcal L$ containing them both.
\item For every nonconstant $3$-term arithmetic progression $(a,b,c)$ in $M$, the points $a,b,c$ lie on a common line in $\mathcal L$ (which is obviously unique, by the previous condition).
\end{itemize}
Then, $\mu_n(M)\leq\lceil n^2/2\rceil$ for every $n$.

Moreover, if $\mu_n(M)=\lceil n^2/2\rceil$, then any set $A$ of $n$ points in $ M$ for which $|\AT M A|=\lceil n^2/2\rceil$ is contained in a line in $\mathcal L$.
\end{proposition}

Note that it immediately follows from Proposition \ref{prop:GSgen} that $\mu_n(M)\leq\lceil n^2/2\rceil$ for every $n$ in case $M$ is an Euclidean or a hyperbolic space of any dimension, by taking $\mathcal L$ to be the family of lines in $M$, in the usual geometric sense.

\begin{proof}[Proof of Proposition \ref{prop:GSgen}]
Let $A$ be a set of $n$ points in $M$. 

If $A$ is contained in a line $L\in\mathcal L$, then since $L$ is isometric to a subset of the real line, it follows from \eqref{eq:R} that $|\AT M A|\leq\left\lceil n^2/2\right\rceil$.

Assume that $A$ is not contained in a line in $\mathcal L$.
For every $L\in{\mathcal L}$, let 
$$r_L:=|A\cap L|.$$
Let
$${\mathcal L}_A:=\{L\in{\mathcal L}\mid r_L\geq 2\}$$ 
be the set of lines in $\mathcal L$ `determined' by the set $A$.
For every $L\in{\mathcal L}_A$, since $L$ is isometric to a subset of the real line, it follows from \eqref{eq:R} that
\begin{equation}\label{eq:GS}
|\AT M {A\cap L}|-r_L\leq\left\lceil\frac{r_L^2}{2}\right\rceil-r_L
=\binom{r_L}{2}-\left\lfloor \frac{r_L}{2}\right\rfloor\leq \binom{r_L}{2}-1.
\end{equation}
For any two distinct points in $A$, there is a unique line in ${\mathcal L}_A$ containing them both. Therefore,
\begin{equation}\label{eq:sum_binom}
\sum_{L\in{\mathcal L}_A}\binom{r_L}{2}=\binom{n}{2},
\end{equation}
and moreover, since the points of $A$ are not all on a single line, 
\begin{equation}\label{eq:Fisher}
|{\mathcal L}_A|\geq n,
\end{equation}
by the de Bruijn--Erd\H{o}s theorem \cite{dBE}.
Finally, for each nonconstant $(a,b,c)\in \AT M A$, there is a unique line in ${\mathcal L}_A$ containing all three points $a,b,c$. Hence,
\begin{equation*}
|\AT M A|-n=\sum_{L\in{\mathcal L}_A}\left(|\AT M {A\cap L}|-r_L\right).
\end{equation*}
Therefore, by \eqref{eq:GS}, \eqref{eq:sum_binom} and \eqref{eq:Fisher},
\begin{equation*}
|\AT M A|-n\leq\sum_{L\in{\mathcal L}_A}\left(\binom{r_L}{2}-1\right)=\binom{n}{2}-|{\mathcal L}_A|\leq \binom{n}{2}-n,
\end{equation*}
and hence, $|\AT M A|\leq\binom{n}{2}<\left\lceil n^2/2\right\rceil$, which concludes the proof.
\end{proof}

We are now ready to prove Theorem \ref{thm:hyper}.

\begin{proof}[Proof of Theorem \ref{thm:hyper}]
Consider the family of `lines'
$${\mathcal L}:=\{\gamma(I)\mid \gamma:I\to M \text{ is a maximal geodesic}\},$$
where a geodesic is maximal if it can not be extended to a geodesic with a larger domain.

Any geodesic segment in $M$ is contained in a unique line in $\mathcal L$. 
Indeed, let $\Gamma$ be a geodesic segment in $M$, let $p,q$ be its endpoints and let $\delta:=d_M(p,q)$. 
There is a unique geodesic $\gamma:[0,\delta]\to M$ such that $\gamma([0,\delta])=\Gamma$, $\gamma(0)=p$ and $\gamma(\delta)=q$. 
Since $M$ is  a Riemannian manifold, $\gamma$ may be uniquely extended to a maximal local geodesic $\hat{\gamma}:I\to M$ (i.e., a local geodesic that can not be extended to a local geodesic with a larger domain). Since each local geodesic in $M$ is a geodesic, it follows that $\hat{\gamma}(I)$ is the unique line in $\mathcal L$ containing the geodesic segment $\Gamma$.

Let us verify that the metric space $M$ and the family $\mathcal L$ of subsets of $M$ satisfy all the conditions of Proposition \ref{prop:GSgen}.
For any geodesic $\gamma:I\to M$, the set $\gamma(I)$ is isometric to the subset $I$ of the real line;
in particular, every line in $\mathcal L$ is isometric to a subset of the real line.
For any two distinct points $p,q$ in $M$, since $M$ is uniquely geodesic, there is a unique geodesic segment $\Gamma$ with endpoints $p,q$; the unique line in $\mathcal L$ containing $\Gamma$ is obviously the unique line in $\mathcal L$ containing both $p$ and $q$.
Finally, if $(a,b,c)$ is a nonconstant $3$-term arithmetic progression in $M$, then $b$ lies on the  geodesic segment $\Gamma$ with endpoints $a,c$, by Lemma \ref{lem:geodesic}; hence, the points $a,b,c$ lie on the line in $\mathcal L$ containing $\Gamma$. 

The result now follows from Proposition \ref{prop:GSgen} and the lower bound in Claim \ref{obs:unigeodesic}.
\end{proof}

\section{The problem in spherical geometry}\label{sec:spheric}

First, we consider the unit circle $S^1=\{u\in{\mathbb R}^2:|u|=1\}$, equipped with the arc length metric.

For every pair of distinct points $a,b$ in $S^1$, let $C_{a,b}$ be the open arc of $S^1$ from $a$ to $b$ counterclockwise, and let $M_{a,b}$ be the midpoint of the arc $C_{a,b}$.

We say that a set $\{p_1,p_2,\ldots,p_n\}$ of $n\geq 2$ points in $S^1$, where the points $p_1,p_2,\ldots,p_n$ are ordered counterclockwise, is \emph{evenly spread around the circle} if 
$$d_{S^1}(p_1,p_2)=\cdots=d_{S^1}(p_{n-1},p_n)=d_{S^1}(p_n,p_1)=2\pi/n.$$ 
Let ${\mathcal F}_0:=\{\emptyset\}$, ${\mathcal F}_1:=\{\{a\}\mid a\in S^1\}$ and for every $n\geq 2$, let ${\mathcal F}_n$ be the family of all $n$-element subsets of $S^1$ that are evenly spread around the circle.
For every positive integer $n$ which is divisible by $4$, let $\rho_n$ be the rotation of $S^1$ by an angle of $\pi/n$ (counterclockwise) and let
\begin{align*}
{\mathcal F}^{[-1]}_n:=&\{A\setminus\{a\}\mid A\in{\mathcal F}_n,\,a\in A\},\\
{\mathcal F}^{[-2]}_n:=&\{A\setminus\{a,b\}\mid A\in{\mathcal F}_n,\,a,b\in A \text{ such that } d_{S^1}(a,b)\leq\textstyle\frac{\pi}{2} \text{ and }M_{a,b},M_{b,a}\in A\},\\
{\mathcal F}^{[+1]}_n:=&\{A\cup\{a\}\mid A\in{\mathcal F}_n,\, \rho_n(a)\in  A\},\\
{\mathcal F}^{[+2]}_n:=&\{A\cup\{a,b\}\mid A\in{\mathcal F}_n,\, a,b\in S^1 \text{ such that } \rho_n(a),\rho_n(b),M_{a,b},M_{b,a}\in A\}.
\end{align*}
(See some examples in Figure \ref{fig1}; note that ${\mathcal F}^{[-2]}_4$ is empty.)

\begin{figure}
\centering
\begin{tikzpicture} 
\draw[color=magenta] (-1,4) circle [radius=1.5];
\draw[fill=none] (-1,4) node {a set in ${\mathcal F}_8$};
\filldraw[blue] (-1,5.5) circle (2pt) {};
\filldraw[blue] (-1,2.5) circle (2pt) {};
\filldraw[blue] (0.5, 4) circle (2pt) {};
\filldraw[blue] (-2.5, 4) circle (2pt) {}; 
\filldraw[blue] (-2.06,2.94) circle (2pt) {};
\filldraw[blue] (-2.06,5.06) circle (2pt) {};
\filldraw[blue] (0.06,2.94) circle (2pt) {};
\filldraw[blue] (0.06,5.06) circle (2pt) {};

\draw[color=magenta] (3,4) circle [radius=1.5];
\draw[fill=none] (3,4) node {a set in ${\mathcal F}^{[+1]}_8$};
\filldraw[blue] (3,5.5) circle (2pt) {};
\filldraw[blue] (3,2.5) circle (2pt) {};
\filldraw[blue] (4.5, 4) circle (2pt) {};
\filldraw[blue] (1.5, 4) circle (2pt) {}; 
\filldraw[blue] (1.94,2.94) circle (2pt) {};
\filldraw[blue] (1.94,5.06) circle (2pt) {};
\filldraw[blue] (4.06,2.94) circle (2pt) {};
\filldraw[blue] (4.06,5.06) circle (2pt) {};
\filldraw[violet] (3.57,5.39) circle (2pt) {};

\draw[color=magenta] (7,4) circle [radius=1.5]; 
\draw[fill=none] (7,4) node {a set in ${\mathcal F}^{[+2]}_8$};
\filldraw[blue] (7,5.5) circle (2pt) {};
\filldraw[blue] (7,2.5) circle (2pt) {};
\filldraw[blue] (8.5, 4) circle (2pt) {};
\filldraw[blue] (5.5, 4) circle (2pt) {};
\filldraw[blue] (5.94,2.94) circle (2pt) {};
\filldraw[blue] (5.94,5.06) circle (2pt) {};
\filldraw[blue] (8.06,2.94) circle (2pt) {};
\filldraw[blue] (8.06,5.06) circle (2pt) {};
\filldraw[violet] (7.57,5.39) circle (2pt) {};
\filldraw[violet] (8.39, 4.57) circle (2pt) {};

\draw[color=magenta] (-1,0) circle [radius=1.5]; 
\draw[fill=none] (-1,-0) node {a set in ${\mathcal F}^{[-2]}_{12}$};
\draw[red] (-1,1.5) circle (2pt) {};
\filldraw[blue] (-1,-1.5) circle (2pt) {};
\filldraw[blue] (0.5,0) circle (2pt) {};
\filldraw[blue] (-2.5, 0) circle (2pt) {};
\filldraw[blue] (-0.25,1.3) circle (2pt) {};
\filldraw[blue] (0.3,0.75) circle (2pt) {};
\filldraw[blue] (-0.25,-1.3) circle (2pt) {};
\filldraw[blue] (0.3,-0.75) circle (2pt) {};
\filldraw[blue] (-1.75,1.3) circle (2pt) {};
\draw[red] (-2.3,0.75) circle (2pt) {};
\filldraw[blue] (-1.75,-1.3) circle (2pt) {};
\filldraw[blue] (-2.3,-0.75) circle (2pt) {};

\draw[color=magenta] (3,0) circle [radius=1.5]; 
\draw[fill=none] (3,0) node {a set in ${\mathcal F}^{[-1]}_{12}$};
\filldraw[blue] (3,1.5) circle (2pt) {};
\filldraw[blue] (3,-1.5) circle (2pt) {};
\filldraw[blue] (4.5, 0) circle (2pt) {};
\filldraw[blue] (1.5, 0) circle (2pt) {};
\filldraw[blue] (3.75,1.3) circle (2pt) {};
\filldraw[blue] (4.3,0.75) circle (2pt) {};
\filldraw[blue] (3.75,-1.3) circle (2pt) {};
\filldraw[blue] (4.3,-0.75) circle (2pt) {};
\filldraw[blue] (2.25,1.3) circle (2pt) {};
\draw[red] (1.7,0.75) circle (2pt) {};
\filldraw[blue] (2.25,-1.3) circle (2pt) {};
\filldraw[blue] (1.7,-0.75) circle (2pt) {};

\draw[color=magenta] (7,0) circle [radius=1.5]; 
\draw[fill=none] (7,0) node {a set in ${\mathcal F}_{12}$};
\filldraw[blue] (7,1.5) circle (2pt) {};
\filldraw[blue] (7,-1.5) circle (2pt) {};
\filldraw[blue] (8.5, 0) circle (2pt) {};
\filldraw[blue] (5.5, 0) circle (2pt) {};
\filldraw[blue] (7.75,1.3) circle (2pt) {};
\filldraw[blue] (8.3,0.75) circle (2pt) {};
\filldraw[blue] (7.75,-1.3) circle (2pt) {};
\filldraw[blue] (8.3,-0.75) circle (2pt) {};
\filldraw[blue] (6.25,1.3) circle (2pt) {};
\filldraw[blue] (5.7,0.75) circle (2pt) {};
\filldraw[blue] (6.25,-1.3) circle (2pt) {};
\filldraw[blue] (5.7,-0.75) circle (2pt) {};
\end{tikzpicture} 
\caption{}\label{fig1}
\end{figure}

The following claim will be used to prove the lower bound in Theorem \ref{thm:S1}.
\begin{claim}\label{claim:ATMn}
For every nonnegative $n$,
\begin{equation}\label{eq:ATMn}
|\AT {S^1} A|=2n\lfloor n/4\rfloor+n\text{ for any } A\in{\mathcal F}_n,
\end{equation}
and for every positive $n$ which is divisible by $4$,
\begin{subequations}
\begin{align}
|\AT {S^1} A|=&\textstyle\frac{1}{2}(n-1)^2+\textstyle\frac{1}{2}(n-1)-1 &\text{ for any } A\in{\mathcal F}^{[-1]}_n,\label{eq:ATMn-1}\\
|\AT {S^1} A|=&\textstyle\frac{1}{2}(n-2)^2+2 &\text{ for any } A\in{\mathcal F}^{[-2]}_n,\label{eq:ATMn-2}\\
|\AT {S^1} A|=& \textstyle\frac{1}{2}(n+1)^2+\textstyle\frac{1}{2}(n+1) &\text{ for any } A\in{\mathcal F}^{[+1]}_n,\label{eq:ATMn+1}\\
|\AT {S^1} A|=&\textstyle\frac{1}{2}(n+2)^2+2 &\text{ for any } A\in{\mathcal F}^{[+2]}_n.\label{eq:ATMn+2}
\end{align}
\end{subequations}
\end{claim}
\begin{proof}
If $A\in{\mathcal F}_n$, then $w_A(a)=2\lfloor n/4\rfloor+1$  for every $a\in A$ and \eqref{eq:ATMn} follows, by \eqref{eq:sum_w}.

For every finite set $A$ of points in $S^1$ and $a\in A$, denote
$$\Gamma_A(a):=\left\{(x,y,z)\in\AT {S^1} A\mid a\in\{x,y,z\}\right\}.$$

Suppose that $n>0$ is divisible by $4$, and let $A\in{\mathcal F}_n$. 
Clearly, using \eqref{eq:ATMn},
$$\sum_{a\in A}\left(|\Gamma_A(a)|-1\right)=3\left(|\AT {S^1} A|-n\right)=\textstyle\frac{3}{2}n^2,$$
and hence, by symmetry, $|\Gamma_A(a)|=\textstyle\frac{3}{2}n+1$ for every $a\in A$.
Therefore, for every $a\in A$, again by \eqref{eq:ATMn},
$$
|\AT {S^1} {A\setminus\{a\}}|=|\AT {S^1} A|-|\Gamma_A(a)|=\left(\textstyle\frac{1}{2}n^2+n\right)-\left(\textstyle\frac{3}{2}n+1\right)=\textstyle\frac{1}{2}(n-1)^2+\textstyle\frac{1}{2}(n-1)-1 .
$$
which proves \eqref{eq:ATMn-1}. Similarly, if $a,b$ are points in $A$ such that $d_{S^1}(a,b)\leq\frac{\pi}{2}$ and $M_{a,b},M_{b,a}\in A$, then
\begin{align*}|\AT {S^1} {A\setminus\{a,b\}}|&=|\AT {S^1} A|-|\Gamma_A(a)|-|\Gamma_A(b)|+|\Gamma_A(a)\cap\Gamma_A(b)|\\
&=\left(\textstyle\frac{1}{2}n^2+n\right)-2\left(\textstyle\frac{3}{2}n+1\right)+6=\textstyle\frac{1}{2}(n-2)^2+2 .
\end{align*}
which proves \eqref{eq:ATMn-2}. 
For every $a\in S^1$ such that $\rho_n(a)\in A$, it holds that $|\Gamma_{A\cup\{a\}}(a)|=w_{A\cup\{a\}}(a)=\frac{1}{2}n+1$ and hence, once more by \eqref{eq:ATMn}, 
$$
|\AT {S^1} {A\cup\{a\}}|=|\AT {S^1} A|+|\Gamma_{A\cup\{a\}}(a)|=\left(\textstyle\frac{1}{2}n^2+n\right)+\left(\textstyle\frac{1}{2}n+1\right)=\textstyle\frac{1}{2}(n+1)^2+\textstyle\frac{1}{2}(n+1) .
$$
which proves \eqref{eq:ATMn+1}. Similarly, if $a,b$ are points in $S^1$ such that $\rho_n(a),\rho_n(b),M_{a,b},M_{b,a}\in A$ (in particular, $a,b$ are not antipodal), then
\begin{align*}|\AT {S^1} {A\cup\{a,b\}}|&=|\AT {S^1} A|+|\Gamma_{A\cup\{a\}}(a)|+|\Gamma_{A\cup\{b\}}(b)|+|\Gamma_{A\cup\{a,b\}}(a)\cap\Gamma_{A\cup\{a,b\}}(b)|\\
&=\left(\textstyle\frac{1}{2}n^2+n\right)+2\left(\textstyle\frac{1}{2}n+1\right)+2=\textstyle\frac{1}{2}(n+2)^2+2 .
\end{align*}
which proves \eqref{eq:ATMn+2}.
\end{proof}

For every $a\in S^1$ denote, for simplicity,
$$H_a:=B_{S^1}(a; \textstyle\frac{\pi}{2})\left(=\{x\in S^1\mid d_{S^1}(a,x)\leq\textstyle\frac{\pi}{2}\}\right)$$
and let $R_a$ be the restriction to $S^1$ of the Euclidean reflection of ${\mathbb R}^2$ through the line connecting $a$ to the origin.

For a finite set $A$ of points in $S^1$, let
$$\Pairs(A):=\left\{\{a,b\}\subseteq A\mid a\neq b,\, \big\lvert\lvert A\cap C_{a,b}\rvert-\lvert A\cap C_{b,a}\rvert\big\rvert\leq 1\right\},$$
and let $\Pairs_0(A)$ be the set of pairs $\{a,b\}$ in $\Pairs(A)$ for which the following conditions hold:
\begin{itemize}
\item The points $a,b$ are antipodal (in particular, $R_a=R_b$).
\item The set $A$ is invariant under the reflection $R_a=R_b$ (in particular, $|A|$ is necessarily even).
\item The points $M_{a,b},M_{b,a}$ are in $A$.
\end{itemize}

For a finite set $A$ of points in $S^1$, recall that for any $b\in A$ we denote by $w_A(b)$ the cardinality of the set $\{(x,y,z)\in\AT {S^1} A\mid y=b\}$, and note that $\frac{1}{2}w_A(a)+\frac{1}{2}w_A(b)$ is an integer for every $a,b\in A$, by Observation \ref{obs:parity}.

\begin{lemma}\label{lem:ww}
Let $A$ be a set of $n$ points in $S^1$ and suppose that $\{a,b\}\in \Pairs(A)$. Then,
\begin{equation}\label{eq:ww_concise}
\textstyle\frac{1}{2}w_A(a)+\frac{1}{2}w_A(b)\leq \left\lfloor\frac{n}{2}\right\rfloor+1.
\end{equation}
Moreover, if $n$ is even and $\frac{1}{2}w_A(a)+\frac{1}{2}w_A(b)=\frac{n}{2}+1$, then $\{a,b\}\in \Pairs_0(A)$.

Finally, if the arc $C_{a,b}$ is longer than the arc $C_{b,a}$ and $\textstyle\frac{1}{2}w_A(a)+\frac{1}{2}w_A(b)\geq \lceil\frac{n}{2}\rceil$, then $A\cap H_a\cap C_{a,b}\subseteq R_a(A)$ and $A\cap H_b\cap C_{a,b}\subseteq R_b(A)$.
\end{lemma}
\begin{proof}
We may assume that the arc $C_{a,b}$ is at least as long as the arc $C_{b,a}$, and that if the two arcs have the same length then $|A\cap C_{a,b}|\leq|A\cap C_{b,a}|$.

If $(b,b,b)\neq (x,b,y)\in \AT {S^1} A$, then $x,y$ are both in $A\cap R_b(A)\cap H_b$ and at least one of them is in the arc $C_{a,b}$. Hence, $w_A(b)\leq 1+2|A\cap R_b(A)\cap H_b\cap C_{a,b}|$. Similarly, $w_A(a)\leq 1+2|A\cap R_a(A)\cap H_a\cap C_{a,b}|$ and hence,
\begin{align}\textstyle\frac{1}{2}w_A(a)+\textstyle\frac{1}{2}w_A(b)&\leq 1+ |A\cap R_a(A)\cap H_a\cap C_{a,b}|+|A\cap R_b(A)\cap H_b\cap C_{a,b}| \nonumber\\
&\leq 1+ |A\cap H_a\cap C_{a,b}|+|A\cap H_b\cap C_{a,b}|\label{eq:ww_complete},
\end{align}
where equality occurs in the second inequality if and only if $A\cap H_a\cap C_{a,b}\subseteq R_a(A)$ and $A\cap H_b\cap C_{a,b}\subseteq R_b(A)$.

If the points $a,b$ are antipodal, then $H_a\cup H_b=S^1$, $ H_a\cap H_b=\{M_{a,b}, M_{b,a}\}$, $|A\cap C_{a,b}|=\lfloor\frac{n-2}{2}\rfloor$ and hence,
\begin{equation}\label{eq:antipodal1}
 |A\cap H_a\cap C_{a,b}|+|A\cap H_b\cap C_{a,b}|=|A\cap C_{a,b}|+|A\cap\{M_{a,b}\}|\leq\left\lfloor\textstyle\frac{n-2}{2}\right\rfloor+1=\left\lfloor\textstyle\frac{n}{2}\right\rfloor.
\end{equation}
If the points $a,b$ are not antipodal, then the set $C_{a,b}\cap H_a\cap H_b$ is empty and hence
\begin{equation}\label{eq:notantipodal1}
 |A\cap H_a\cap C_{a,b}|+|A\cap H_b\cap C_{a,b}|\leq |A\cap C_{a,b}|\leq\left\lceil\textstyle\frac{n-2}{2}\right\rceil=\left\lceil\textstyle\frac{n}{2}\right\rceil-1.\end{equation}
Combining \eqref{eq:ww_complete}, \eqref{eq:notantipodal1} and \eqref{eq:antipodal1} yields \eqref{eq:ww_concise}.

Suppose now that $n$ is even and $\frac{1}{2}w_A(a)+\frac{1}{2}w_A(b)=\frac{n}{2}+1$. By \eqref{eq:ww_complete} and \eqref{eq:notantipodal1}, the points $a,b$ are necessarily antipodal. Then, equality must occur in \eqref{eq:antipodal1}, i.e., $M_{a,b}\in A$, and the inequalities in \eqref{eq:ww_complete} are necessarily equalities as well; in particular, $A\cap H_a\cap C_{a,b}\subseteq R_a(A)$ and $A\cap H_b\cap C_{a,b}\subseteq R_b(A)$.
Therefore, since $H_a\cup H_b=S^1$ and $R_a=R_b$, 
$$
A\cap C_{a,b}=\left(A\cap H_a\cap C_{a,b}\right)\cup\left(A\cap H_b\cap C_{a,b}\right)\subseteq R_a(A)\cup R_b(A)=R_b(A),
$$
and hence, since $C_{a,b}=R_b(C_{b,a})$,
$$A\cap C_{a,b}\subseteq R_b(A)\cap C_{a,b}=R_b(A\cap C_{b,a}).$$
Therefore, since $|A\cap C_{a,b}|=|A\cap C_{b,a}|=|R_b(A\cap C_{b,a})|$, necessarily $A\cap C_{a,b}=R_b(A\cap C_{b,a})$, i.e., the set $A$ is invariant under the reflection $R_a=R_b$.
In particular, since $M_{a,b}\in A$, it also holds that $M_{b,a}=R_b(M_{a,b})\in A$.
In conclusion, $\{a,b\}\in \Pairs_0(A)$.

Finally, suppose that the arc $C_{a,b}$ is longer than the arc $C_{b,a}$ and that $\textstyle\frac{1}{2}w_A(a)+\frac{1}{2}w_A(b)\geq \lceil\frac{n}{2}\rceil$. 
In particular, the points $a,b$ are not antipodal. Then, by \eqref{eq:notantipodal1}, the inequalities in \eqref{eq:ww_complete} are necessarily equalities. In particular, $A\cap H_a\cap C_{a,b}\subseteq R_a(A)$ and $A\cap H_b\cap C_{a,b}\subseteq R_b(A)$.
\end{proof}

\begin{obs}\label{obs:evenly}
Let $A$ be a finite subset of $S^1$ such that $r:=|\Pairs_0(A)|>1$.
Then, 
$$\Pairs_0(A)=\left\{\{p_1,p_{r+1}\},\{p_2,p_{r+2}\},\ldots, \{p_r,p_{2r}\}\right\},$$
where the points $p_1,p_2,\ldots,p_{2r}$ are ordered counterclockwise and the set $\{p_1,p_2,\ldots,p_{2r}\}$ is evenly spread around the circle.
\end{obs}
\begin{proof}
Since $\Pairs_0(A)$ is a set of $r$ pairs of antipodal points, it follows that  
$$\Pairs_0(A)=\left\{\{p_1,p_{r+1}\},\{p_2,p_{r+2}\},\ldots, \{p_r,p_{2r}\}\right\},$$
where the points $p_1,p_2,\ldots,p_{2r}$ are ordered counterclockwise.
Observe that if the set $A$ is invariant under some isometry $\rho$ of $S^1$, then by the definition of $\Pairs_0(A)$, the set $\{p_1,p_2,\ldots,p_{2r}\}$ is invariant under $\rho$ as well.
If the set $\{p_1,p_2,\ldots,p_{2r}\}$ is not evenly spread around the circle, then there is $1\leq i\leq r$ such that $d_{S^1}(p_i,p_{i+1})\neq d_{S^1}(p_{i+1},p_{i+2})$. With no loss of generality, assume that  $d_{S^1}(p_i,p_{i+1})< d_{S^1}(p_{i+1},p_{i+2})$. Then, $R_{p_{i+1}}(p_i)$ is in the open arc $C_{p_{i+1},p_{i+2}}$ and hence $R_{p_{i+1}}(p_i)\notin\{p_1,p_2,\ldots,p_{2r}\}$. Therefore, the set $\{p_1,p_2,\ldots,p_{2r}\}$ is not invariant under the reflection $R_{p_{i+1}}$, contradicting the invariance of $A$ under $R_{p_{i+1}}$.
\end{proof}

\begin{lemma}\label{lem:structure}
Let $A$ be a set of $n$ points in $S^1$ such that $n\bmod 4=2$ and $|\Pairs_0(A)|>1$. 
There is a set $\Pairs_1\subseteq \Pairs(A)\setminus \Pairs_0(A)$ such that $|\Pairs_1|\geq |\Pairs_0(A)|$ and for every $\{a,b\}\in \Pairs_1$,
$$\textstyle\frac{1}{2}w_A(a)+\frac{1}{2}w_A(b)\leq \frac{n}{2}-1.$$
\end{lemma}

\begin{proof}
Denote  $r:=|\Pairs_0(A)|$. 
By Observation \ref{obs:evenly}, 
$$\Pairs_0(A)=\left\{\{p_1,p_{r+1}\},\{p_2,p_{r+2}\},\ldots, \{p_r,p_{2r}\}\right\},$$
where the points $p_1,p_2,\ldots,p_{2r}$ are ordered counterclockwise and the set $\{p_1,p_2,\ldots,p_{2r}\}$ is evenly spread around the circle.  

For the rest of the proof we interpret the  indices of points in $S^1$ cyclically (e.g.,  $p_{2r+1}$ is interpreted as $p_1$).  
Let $m:=|A\cap C_{p_1,p_2}|$. Since the set $A$ is invariant under the reflection $R_{p_i}$ for every $1\leq i\leq 2r$, it follows that for every $1\leq i\leq 2r$, 
$$|A\cap C_{p_i,p_{i+1}}|=m.$$
Hence, $n=|A|=2r(m+1)$. Since $n\bmod 4=2$, it follows that $m$ is even and $r$ is odd, i.e., $r=2s+1$ for some integer $s$ which is necessarily positive since $r>1$.

For every $1\leq i\leq 2r$, since $\{p_{i-s},p_{i+1+s}\}\in \Pairs_0(A)$, the point $a_{i}:=M_{p_i,p_{i+1}}=M_{p_{i-s},p_{i+1+s}}$ is in $A$.
Let $m_0:=|A\cap C_{p_1,a_1}|$ and $m_1:=|A\cap C_{a_1,p_2}|$.
Since the set $A$ is invariant under the reflection $R_{p_i}$ for every $1\leq i\leq 2r$, it follows that for every $1\leq i\leq 2r$,
$$|A\cap C_{p_i,a_i}|=m_{(i+1)\bmod 2},\quad |A\cap C_{a_i,p_{i+1}}|=m_{i\bmod 2}.$$

For every $1\leq i\leq 2r$, let $b_i$ be the point in $A$ for which $\{a_i,b_i\}\in \Pairs(A)$ (see an example, in which $r=3, m_0=0, m_1=1$, in Figure \ref{fig2}).

\begin{figure}
\centering
\begin{tikzpicture}
\draw[color=magenta] (0,0) circle [radius=4]; 
\filldraw[blue] (0,4) circle (2pt) {};
\draw[fill=none] (0,3.6) node {$p_2$};
\filldraw[blue] (0,-4) circle (2pt) {};
\draw[fill=none] (0,-3.6) node {$p_5$};
\filldraw[blue] (4, 0) circle (2pt) {};
\draw[fill=none] (3.6,0) node {$a_6$};
\filldraw[blue] (-4, 0) circle (2pt) {};
\draw[fill=none] (-3.6,0) node {$a_3$};
\filldraw[blue] (2,3.5) circle (2pt) {};
\draw[fill=none] (1.8,3.15) node {$a_1$};
\filldraw[blue] (3.5,2) circle (2pt) {};
\draw[fill=none] (3.15,1.8) node {$p_1$};
\filldraw[blue] (2,-3.5) circle (2pt) {};
\draw[fill=none] (1.8,-3.15) node {$a_5$};
\filldraw[blue] (3.5,-2) circle (2pt) {};
\draw[fill=none] (3.15,-1.8) node {$p_6$};
\filldraw[blue] (-2,3.5) circle (2pt) {};
\draw[fill=none] (-1.8,3.15) node {$a_2$};
\filldraw[blue] (-3.5,2) circle (2pt) {};
\draw[fill=none] (-3.15,1.8) node {$p_3$};
\filldraw[blue] (-2,-3.5) circle (2pt) {};
\draw[fill=none] (-1.8,-3.15) node {$a_4$};
\filldraw[blue] (-3.5,-2) circle (2pt) {};
\draw[fill=none] (-3.15,-1.8) node {$p_4$};
\filldraw[blue] (1.5,3.7) circle (2pt) {};
\draw[fill=none] (1.35,3.4) node {$b_4$};
\filldraw[blue] (-1.5,3.7) circle (2pt) {};
\draw[fill=none] (-1.35,3.4) node {$b_5$};
\filldraw[blue] (2.45,-3.15) circle (2pt) {};
\draw[fill=none] (2.2,-2.85) node {$b_2$};
\filldraw[blue] (3.95,-0.55) circle (2pt) {};
\draw[fill=none] (3.55,-0.5) node {$b_3$};
\filldraw[blue] (-2.45,-3.15) circle (2pt) {};
\draw[fill=none] (-2.2,-2.85) node {$b_1$};
\filldraw[blue] (-3.95,-0.55) circle (2pt) {};
\draw[fill=none] (-3.55,-0.5) node {$b_6$};
\end{tikzpicture} 
\caption{}\label{fig2}
\end{figure}

For every $1\leq i\leq 2r$, clearly $a_i\notin\{p_1,p_2,,\ldots,p_{2r}\}$ and hence $\{a_i,b_i\}\notin \Pairs_0(A)$. 
Therefore, 
$$\Pairs_1:=\left\{\{a_1,b_1\},\{a_2,b_2\},\ldots,\{a_{2r},b_{2r}\}\right\}\subseteq \Pairs(A)\setminus \Pairs_0(A)$$
and clearly $|\Pairs_1|\geq r=|\Pairs_0(A)|$.
(We remark that it is easy to show that for every $1\leq i\leq 2r$, the point $b_i$ is in the open arc $C_{p_{i+r},p_{i+1+r}}$ and $b_i\neq a_{i+r}$, implying that the $2r$ pairs $\{a_1,b_1\},\{a_2,b_2\},\ldots,\{a_{2r},b_{2r}\}$ are distinct and hence $|\Pairs_1|\geq 2r$.)

To conclude the proof, we will show that for every $1\leq i\leq 2r$, 
$$\textstyle\frac{1}{2}w_A(a_i)+\frac{1}{2}w_A(b_i)\leq \frac{n}{2}-1.$$

Since $m=m_0+m_1+1$ is even, necessarily $m_0\neq m_1$. With no loss of generality, assume that $m_0<m_1$.

Assume first that $1\leq i\leq 2r$ is even.
Then,
$$
|A\cap C_{a_i,a_{i+r}}|=r+(r-1)m+2m_0<r+(r-1)m+m_0+m_1=r(m+1)-1=\textstyle\frac{1}{2}(n-2)=|A\cap C_{a_i,b_i}|.
$$
Hence, the point $a_{i+r}$ is in the arc $C_{a_i,b_i}$.
Therefore, since the points $a_i,a_{i+r}$ are antipodal, the arc $C_{a_i,b_i}$ is necessarily longer than the arc $C_{b_i,a_i}$.
Additionally, since $C_{a_{i+1},p_{i+2}}=R_{a_i}(C_{p_{i-1},a_{i-1}})$,
$$|R_{a_i}(A)\cap C_{a_{i+1},p_{i+2}}|=|A\cap C_{p_{i-1},a_{i-1}}|=m_0<m_1=|A\cap C_{a_{i+1},p_{i+2}}|$$
and hence $A\cap C_{a_{i+1},p_{i+2}}\not\subseteq R_{a_i}(A)$.
Surely, $C_{a_{i+1},p_{i+2}}\subset C_{a_i,a_{i+r}}\subset C_{a_i,b_i}$, since $r>1$ and the point $a_{i+r}$ is in the arc $C_{a_i,b_i}$. 
Moreover, $C_{a_{i+1},p_{i+2}}\subset H_{a_i}$, since $r\geq 3$.
Hence, $C_{a_{i+1},p_{i+2}}\subset H_{a_i}\cap C_{a_i,b_i}$ and therefore, $A\cap H_{a_i}\cap C_{a_i,b_i}\not\subseteq R_{a_i}(A)$. 
Hence, by the last part of Lemma \ref{lem:ww},
\begin{equation*}
\textstyle\frac{1}{2}w_A(a_i)+\frac{1}{2}w_A(b_i)\leq \frac{n}{2}-1.
\end{equation*}
Similarly, if $1\leq i\leq 2r$ is odd, then the arc $C_{b_i,a_i}$ is longer than the arc $C_{a_i,b_i}$ and $A\cap H_{a_i}\cap C_{b_i,a_i}\not\subseteq R_{a_i}(A)$ (since $A\cap C_{p_{i-1},a_{i-1}}\not\subseteq R_{a_i}(A)$) and hence, 
\begin{equation*}
\textstyle\frac{1}{2}w_A(a_i)+\frac{1}{2}w_A(b_i)=\textstyle\frac{1}{2}w_A(b_i)+\frac{1}{2}w_A(a_i)\leq \frac{n}{2}-1.
\qedhere\end{equation*}
\end{proof}

By combining \eqref{eq:sum_w}, Lemma \ref{lem:ww} and Lemma \ref{lem:structure}, we may now prove the following proposition that will be used to prove the upper bound in Theorem \ref{thm:S1}.
\begin{proposition}\label{prop:crudeS1}
Let $A$ be a set of  $n$ points in $S^1$. Then,
\begin{equation}\label{eq:generalupperS1}
|\AT {S^1} A|\leq n\left\lfloor n/2\right\rfloor+n.
\end{equation}
Moreover, if $n\bmod 4=2$, then
\begin{equation}\label{eq:rotation}
|\AT {S^1} A|\leq\textstyle\frac{1}{2}n^2+2.
\end{equation}
\end{proposition}
\begin{proof}
For simplicity, denote 
$$\nu:=\begin{cases}
1 & n \text{ is even},\\
2 & n \text{ is odd}.
\end{cases}$$
Clearly, for every $a\in A$, 
\begin{equation}\label{eq:aS}
|\{b\in A:\{a,b\}\in \Pairs(A)\}|=\nu,
\end{equation}
and hence 
\begin{equation}\label{eq:S}|\Pairs(A)|=\textstyle\frac{\nu}{2}n.\end{equation}
By \eqref{eq:sum_w}, in light of \eqref{eq:aS}, 
\begin{equation}\label{eq:sum}|\AT {S^1} A|=\textstyle\frac{2}{\nu}\sum_{\{a,b\}\in \Pairs(A)}\left(\textstyle\frac{1}{2}w_A(a)+\frac{1}{2}w_A(b)\right),\end{equation}
and \eqref{eq:generalupperS1} follows, by using \eqref{eq:S} and \eqref{eq:ww_concise}.

Suppose now that $n\bmod 4=2$. Then, $\nu=1$; hence, $|\Pairs(A)|=n/2$, by \eqref{eq:S}; moreover, 
$$
|\AT {S^1} A|=2\sum_{\{a,b\}\in \Pairs(A)}\left(\textstyle\frac{1}{2}w_A(a)+\frac{1}{2}w_A(b)\right),
$$
by \eqref{eq:sum}. Therefore, if $|\Pairs_0(A)|\leq1$, then by Lemma \ref{lem:ww},
$$
|\AT {S^1} A|\leq2 |\Pairs_0(A)|\left(\textstyle\frac{n}{2}+1\right)+2|\Pairs(A)\setminus \Pairs_0(A)|\textstyle\frac{n}{2}=n|\Pairs(A)|+2|\Pairs_0(A)|\leq \textstyle\frac{1}{2}n^2+2,
$$
and if $|\Pairs_0(A)|>1$ then, by Lemma \ref{lem:structure}, there is a set $\Pairs_1\subseteq \Pairs(A)\setminus \Pairs_0(A)$ such that $|\Pairs_1|\geq |\Pairs_0(A)|$ and 
$\textstyle\frac{1}{2}w_A(a)+\frac{1}{2}w_A(b)\leq \frac{n}{2}-1$ for every $\{a,b\}\in \Pairs_1$
and hence, by Lemma \ref{lem:ww},
\begin{align*}|\AT {S^1} A|&\leq 2|\Pairs_0(A)|\left(\textstyle\frac{n}{2}+1\right)+2|\Pairs_1|\left(\textstyle\frac{n}{2}-1\right)+2|\Pairs(A)\setminus(\Pairs_0(A)\cup \Pairs_1)|\textstyle\frac{n}{2}\\
&=n|\Pairs(A)|-2\left(|\Pairs_1|-|\Pairs_0(A)|\right)\leq \textstyle\frac{1}{2}n^2.
\qedhere\end{align*}
\end{proof}

We are finally ready to prove Theorem \ref{thm:S1}.
\begin{proof}[Proof of Theorem \ref{thm:S1}]
Recall that we need to show that for every $n\neq 2$,
$$\mu_n(S^1)=\begin{cases}
 \frac{1}{2}n^2+n & n\bmod 4=0,\\
\frac{1}{2}n^2+\frac{1}{2}n& n\bmod 4=1,\\
\frac{1}{2}n^2+2 & n\bmod 4=2,\\
\frac{1}{2}n^2+\frac{1}{2}n-1 & n\bmod 4=3.
\end{cases}$$

Suppose first that $n\bmod 4=0$. 
The lower bound follows since $|\AT {S^1} A|=\frac{1}{2}n^2+n$ for any $A\in {\mathcal F}_n$, by \eqref{eq:ATMn}.
The upper bound follows immediately from \eqref{eq:generalupperS1}.

Suppose now that $n\bmod 4=1$. 
The lower bound follows since $|\AT {S^1} A|=\frac{1}{2}n^2+\frac{1}{2}n$ for any $A\in {\mathcal F}_n\cup{\mathcal F}^{[+1]}_{n-1}$, by \eqref{eq:ATMn} and \eqref{eq:ATMn+1}.
The upper bound follows immediately from \eqref{eq:generalupperS1}.

Next, suppose that $n\bmod 4=2$. 
The lower bound follows since $|\AT {S^1} A|=\frac{1}{2}n^2+2$ for any $A\in {\mathcal F}^{[-2]}_{n+2}\cup{\mathcal F}^{[+2]}_{n-2}$, by \eqref{eq:ATMn-2} and \eqref{eq:ATMn+2}.
The upper bound follows immediately from \eqref{eq:rotation}.

Finally, suppose that $n\bmod 4=3$. 
The lower bound follows since $|\AT {S^1} A|=\frac{1}{2}n^2+\frac{1}{2}n-1$ for any $A\in {\mathcal F}^{[-1]}_{n+1}$, by \eqref{eq:ATMn-1}.
To prove the upper bound, let $A$ be a set of $n$ points in $S^1$. 
By \eqref{eq:generalupperS1}, it holds that $|\AT {S^1} A|\leq n\frac{n-1}{2}+n=\frac{1}{2}n^2+\frac{1}{2}n$,
and the result follows since $|\AT {S^1} A|-n$ is even, by Observation \ref{obs:parity}, whereas $n\frac{n-1}{2}$ is odd.
\end{proof}

\begin{rem}
Let $n$ be a positive integer which is divisible by $4$. It was mentioned in the proof of Theorem \ref{thm:S1} that $|\AT {S^1} A|=\mu_n(S^1)$ for every $A\in {\mathcal F}_n$. In fact,
$$\{A\subset S^1: |A|=n,\,|\AT {S^1} A|=\mu_n(S^1)\}={\mathcal F}_n.$$
Indeed, let $A$ be a set of $n$ points in $S^1$ such that $|\AT {S^1} A|=\mu_n(S^1)$, i.e., $|\AT {S^1} A|=\frac{1}{2}n^2+n$. 
Then, the proof of Proposition \ref{prop:crudeS1} reveals that for every $\{a,b\}\in \Pairs(A)$, it holds that $\frac{1}{2}w_A(a)+\frac{1}{2}w_A(b)=\frac{n}{2}+1$ and hence $\{a,b\}\in \Pairs_0(A)$, by Lemma \ref{lem:ww}. 
Therefore, $\Pairs(A)=\Pairs_0(A)$ and it follows, by Observation \ref{obs:evenly}, that $A$ is evenly spread around the circle.
\end{rem}

We were not able to determine, for general $n$, the maximal number of $3$-term arithmetic progressions in $n$-element subsets of the $2$-dimensional sphere $S^2=\{u\in{\mathbb R}^3: |u|=1\}$. 
We believe that the maximum is attained for (appropriate) sets that are contained in the union of a \emph{great circle} of the sphere and the pair of respective `poles'.
We proceed to find the maximal number of $3$-term arithmetic progressions in sets of this form. 
Let 
$${\mathcal P}:=\left\{(1,0,0),(-1,0,0)\right\}$$
be the pair of `north and south pole' of the sphere $S^2$, let
$${\mathcal E}_0:=\left\{(x,y,z)\in S^2\mid z=0\right\}=\left\{(x,y,0)\mid (x,y)\in S^1\right\}$$
be the corresponding `equator' and consider the set
$${\mathcal E}:=\left\{(x,y,z)\in S^2\mid z\in\{-1,0,1\}\right\}={\mathcal E}_0\cup{\mathcal P}$$
(with respect to the metric of the sphere $S^2$).

\begin{proposition}\label{prop:equator}
For every $n\geq 2$,
\begin{equation}\label{eq:S2}
\mu_n({\mathcal E})=\frac{1}{2}n^2+
\begin{cases}
2n-4 & n\bmod 4=0,\\
\frac{5}{2}n-8& n\bmod 4=1,\\
3n-6 & n\bmod 4=2,\\
\frac{5}{2}n-7 & n\bmod 4=3.
\end{cases}
\end{equation}
\end{proposition}
In particular, \eqref{eq:S2} implies, in light of Theorem \ref{thm:S1}, that $\mu_n({\mathcal E})>\mu_n(S^1)$ for every $n\geq 5$, thus confirming \eqref{eq:S2S1}, since for every $n$, obviously
\begin{equation}\label{eq:S2E}
\mu_n(S^2)\geq\mu_n({\mathcal E}).
\end{equation}
As already mentioned, we believe that the lower bound \eqref{eq:S2E} for $\mu_n(S^2)$ is tight for every $n$, i.e., we make the following conjecture.
\begin{conj}\label{conj:S2}
For every $n\geq 2$
$$\mu_n(S^2)=\frac{1}{2}n^2+
\begin{cases}
2n-4 & n\bmod 4=0,\\
\frac{5}{2}n-8& n\bmod 4=1,\\
3n-6 & n\bmod 4=2,\\
\frac{5}{2}n-7 & n\bmod 4=3.
\end{cases}$$
\end{conj}
Note that for every $n$, the upper bound $\mu_n(S^2)\leq n^2$ trivially holds, since for any two points $a,b\in S^2$ there is a unique point $c\in S^2$ such that $(a,b,c)$ is a $3$-term arithmetic progression.

\begin{proof}[Proof of Proposition \ref{prop:equator}]
Clearly $\mu_2(S^2)=2$, hence we assume that $n\geq 3$.
For simplicity, denote by $\mu^*_n$ the right hand side of \eqref{eq:S2}.
By Theorem \ref{thm:S1} (and since $\mu_2(S^1)=2$), 
\begin{subequations} 
\begin{align}
\mu^*_n&\geq\mu_n(S^1),\label{eq:star}\\
\mu^*_n&\geq\mu_{n-1}(S^1)+n,\label{eq:star-1}\\
\mu^*_n&=\mu_{n-2}(S^1)+\begin{cases}
4n-8 & n\bmod 4\neq 2 \text{ and } n\neq 4,\\
4n-6 & n\bmod 4=2 \text{ or } n=4.
\end{cases}\label{eq:star-2}
\end{align}
\end{subequations}
For any $(x,y)\in S^1$, let $\nu(x,y):=(-x,-y)$ and let $\iota(x,y):=(x,y,0)\in{\mathcal E}_0$.
Let $A$ be a set of $n$ points in $\mathcal E$. For simplicity, denote $A_0:=\{u\in S^1\mid \iota(u)\in A\}$, $A_{\pm}:=A\cap\mathcal P$.
Note that
\begin{align*}
\AT {\mathcal E} A&\subset {\mathcal E}_0^3\cup\left({\mathcal E}_0\times {\mathcal P}\times {\mathcal E}_0\right)\cup\left({\mathcal P}\times {\mathcal E}_0\times {\mathcal P}\right)\cup\{(u,u,u)\mid u\in{\mathcal P}\},\\
\AT {\mathcal E} A\cap\mathcal E_0^3&=\AT {{\mathcal E}_0} {\iota(A_0)}=\{(\iota(a),\iota(b),\iota(c))\mid (a,b,c)\in\AT {S^1} {A_0}\},\\
\AT {\mathcal E} A\cap\left({\mathcal E}_0\times {\mathcal P}\times {\mathcal E}_0\right)&=\{(\iota(a),u,\iota(\nu(a)))\mid a\in A_0\cap\nu\left(A_0\right),u\in A_{\pm}\},
\end{align*}
and
$$\AT {\mathcal E} A\cap\left({\mathcal P}\times {\mathcal E}_0\times {\mathcal P}\right)=\begin{cases}
\emptyset &  {\mathcal P}\not\subset A,\\
\{(u,\iota(a),-u)\mid u\in{\mathcal P},\, a\in A_0\} & {\mathcal P}\subset A.
\end{cases}$$
Therefore,
\begin{equation}\label{eq:PsubA}
|\AT {\mathcal E} A|=
\begin{cases}
|\AT {S^1} {A_0}|+|A_{\pm}|\cdot|A_0\cap\nu\left(A_0\right)|+|A_{\pm}| &  {\mathcal P}\not\subset A,\\
|\AT {S^1} {A_0}|+2|A_0\cap\nu\left(A_0\right)|+2|A_0|+2 & {\mathcal P}\subset A.
\end{cases}
\end{equation}
Now, if $A_{\pm}=\emptyset$, then $|\AT {\mathcal E} A|=|\AT {S^1} {A_0}|\leq\mu_n(S^1)\leq\mu^*_n$, by \eqref{eq:star}; 
if $|A_{\pm}|=1$, then $|\AT {\mathcal E} A|\leq \mu_{n-1}(S^1)+2\lfloor(n-1)/2\rfloor+1\leq\mu^*_n$,
by \eqref{eq:PsubA} and \eqref{eq:star-1};
if ${\mathcal P}\subset A$ and $n$ is not divisible by $4$ then, by \eqref{eq:PsubA} and \eqref{eq:star-2},
\begin{equation}\label{eq:S2bmodn0}
|\AT {\mathcal E} A|\leq \mu_{n-2}(S^1)+2\cdot 2\lfloor(n-2)/2\rfloor+2(n-2)+2=\mu^*_n;
\end{equation}
and if ${\mathcal P}\subset A$, $n$ is divisible by $4$ and $\nu(A_0)\neq A_0$ then, again by \eqref{eq:PsubA} and \eqref{eq:star-2},
\begin{equation}\label{eq:S2bmod0nantipod}
|\AT {\mathcal E} A|\leq \mu_{n-2}(S^1)+2(n-4)+2(n-2)+2\leq \mu^*_n-2.
\end{equation}
Finally, suppose that ${\mathcal P}\subset A$, $n$ is divisible by $4$ and $\nu(A_0)=A_0$. 
If $\{a,b\}\in \Pairs_0(A_0)$, then $|A_0\cap C_{a,M_{a,b}}|=|A_0\cap C_{M_{a,b},b}|$, since the set $A_0$ is invariant under both the antipodal map $\nu$ and the reflection $R_a=R_b$, and hence, since $M_{a,b}\in A_0$, it follows that 
$$\textstyle\frac{(n-2)-2}{2}=|A_0\cap C_{a,b}|=|A_0\cap C_{a,M_{a,b}}|+1+|A_0\cap C_{M_{a,b},b}|$$
is odd, contradicting the assumption that $n$ is divisible by $4$.  
Therefore, $\Pairs_0(A_0)$ is necessarily empty and hence, 
\begin{equation}\label{eq:S1bmod2antipod}
|\AT {S^1} {A_0}|\leq\textstyle\frac{1}{2}(n-2)^2, 
\end{equation}
by \eqref{eq:sum_w} and Lemma \ref{lem:ww}. 
Therefore, by \eqref{eq:PsubA},
\begin{equation}\label{eq:S2bmod0}
|\AT {\mathcal E} A|\leq \textstyle\frac{1}{2}(n-2)^2+2(n-2)+2(n-2)+2=\mu^*_n.
\end{equation}

To conclude the proof, note that by using Claim \ref{claim:ATMn}, equality occurs in \eqref{eq:S2bmodn0} whenever 
\begin{equation*}A_0\in\begin{cases}
{\mathcal F}^{[-1]}_{n-1} & n\bmod 4=1,\\
{\mathcal F}_{n-2} & n\bmod 4=2 \text{ or } n=3, \\
{\mathcal F}^{[+1]}_{n-3} & n\bmod 4=3 \text{ and } n\neq 3,
\end{cases}\end{equation*}
and for every $n$ divisible by $4$, equality occurs in \eqref{eq:S2bmod0} whenever $A_0\in{\mathcal F}_{n-2}$.
\end{proof}

\begin{rem}
It is not true in general that if $A_0$ is a set of $n-2$ points in $S^1$ for which $|\AT {S^1} {A_0}|=\mu_{n-2}(S^1)$, then $|\AT {\mathcal E} {\iota(A_0)\cup{\mathcal P}}|=\mu_{n}({\mathcal E})$ (where $\iota$ is as in the proof of Proposition \ref{prop:equator}).
Indeed, if $n\bmod 4=3$, then for any $A_0\in{\mathcal F}_{n-2}$, it was mentioned in the proof of Theorem \ref{thm:S1} that $|\AT {S^1} {A_0}|=\mu_{n-2}(S^1)$, but the set $A_0\cap\nu\left(A_0\right)$ is empty and hence $|\AT {\mathcal E} {\iota(A_0)\cup{\mathcal P}}|=\mu_{n}({\mathcal E})-2(n-3)$, by \eqref{eq:PsubA} and Proposition \ref{prop:equator}.

Moreover,  if $n>4$ is divisible by $4$, then there is not a single set $A_0$ of $n-2$ points in $S^1$ for which both $|\AT {S^1} {A_0}|=\mu_{n-2}(S^1)$ and $|\AT {\mathcal E} {\iota(A_0)\cup{\mathcal P}}|=\mu_{n}({\mathcal E})$.
Indeed, if $\nu(A_0)= A_0$ then $|\AT {S^1} {A_0}|\leq\mu_{n-2}(S^1)-2$, by \eqref{eq:S1bmod2antipod} and \eqref{eq:ATMn-2} (or \eqref{eq:ATMn+2}), 
and if $\nu(A_0)\neq A_0$ then $|\AT {\mathcal E} {\iota(A_0)\cup{\mathcal P}}|\leq\mu_{n}({\mathcal E})-2$, by \eqref{eq:S2bmod0nantipod} and the last line of the proof of Proposition \ref{prop:equator}.
\end{rem}

\section{Additional bounds}\label{sec:misc}
\begin{proposition}\label{prop:trees}
Let $\mathbb T_r$ be the (discrete) $r$-regular tree, equipped with the graph metric. For every ball $A$ in $\mathbb T_r$ it holds that
\begin{equation}\label{eq:tree_exact}
|\AT {\mathbb T_r} A|=\left(\frac{1}{2}+\frac{(r-2)^2}{2r^2}\right)|A|^2+\frac{2(r-2)}{r^2}|A|+\frac{2}{r^2}.
\end{equation}
Consequently,
$$\limsup_{n\to\infty}\frac{\mu_n({\mathbb T}_r)}{n^2}\geq\frac{1}{2}+\frac{(r-2)^2}{2r^2}.$$
\end{proposition}
Note that for any tree $T$, it holds that $\mu_n(T)\leq n^2-2n+2$ for every $n$, by Claim \ref{obs:unique}.

\begin{proof}[Proof of Proposition \ref{prop:trees}]
Consider a ball $A=B_{{\mathbb T}_r}(v_0; d_0)$ for some vertex $v_0$ and a nonnegative integer $d_0$. Note that
\begin{equation}\label{eq:gsum}
|A|=1+r\sum_{d=1}^{d_0}(r-1)^{d-1}=1+\frac{r}{r-2}\left((r-1)^{d_0}-1\right)=\frac{r}{r-2}\left((r-1)^{d_0}-\frac{2}{r}\right).
\end{equation}
Let $b$ be a vertex in $A$, let $d_1:=d_0-d_{\mathbb T_r}(v_0,b)$, and let $\mathcal C$ be the collection of connected components of the graph that is obtained from $\mathbb T_r$ by removing the vertex $b$ (and all edges adjacent to it). 
For $a,c\in \mathbb T_r$ it holds that $(a,b,c)\in \AT {\mathbb T_r} A$ if and only if either $a=b=c$ or $a,c$ belong to different sets in the collection $\mathcal C$ and $d_{\mathbb T_r}(b,a)=d_{\mathbb T_r}(b,c)\leq d_1$.
Therefore,
\begin{align*}
w_A(b)&=1+|\mathcal C|(|\mathcal C|-1)\sum_{d=1}^{d_1}\left((r-1)^{d-1}\right)^2=1+r(r-1)\sum_{d=1}^{d_1}\left((r-1)^2\right)^{d-1}\\
&=1+r(r-1)\frac{\left((r-1)^2\right)^{d_1}-1}{(r-1)^2-1}=\frac{(r-1)^{2d_1+1}-1}{r-2}.
\end{align*}
Hence, by \eqref{eq:sum_w},
\begin{align*}
(r-2)|\AT {\mathbb T_r} A|+|A|&=\sum_{b\in A}\left((r-2)w_A(b)+1\right)\\
&=\left((r-2)w_A(v_0)+1\right)+\sum_{d_1=0}^{d_0-1}\left(\sum_{b\in S_{{\mathbb T}_r}(v_0; d_0-d_1)}\left((r-2)w_A(b)+1\right)\right)\\
&=(r-1)^{2d_0+1}+\sum_{d_1=0}^{d_0-1}r(r-1)^{d_0-d_1-1} (r-1)^{2d_1+1}.
\end{align*}
Therefore, by \eqref{eq:gsum},
\begin{align*}
(r-2)|\AT {\mathbb T_r} A|+|A|&=(r-1)^{d_0}\left((r-1)\cdot(r-1)^{d_0}+r\sum_{d_1=0}^{d_0-1}(r-1)^{d_1}\right)\\
&=\left(\frac{r-2}{r}|A|+\frac{2}{r}\right)\left((r-1)\left(\frac{r-2}{r}|A|+\frac{2}{r}\right)+(|A|-1)\right),
\end{align*}
which implies \eqref{eq:tree_exact}.
\end{proof}

\begin{proposition}\label{prop:lattices}
Consider the $\ell$-dimensional lattice graph ${\mathbb Z}^{\ell}$ (where two vertices are adjacent if the Euclidean distance between them is $1$), with respect to the graph metric (alternatively, $d_{{\mathbb Z}^{\ell}}(a,b)=\sum_{i=1}^{\ell}|a_i-b_i|$ for every $a=(a_i)_{i=1}^{\ell}$ and $b=(b_i)_{i=1}^{\ell}$ in ${\mathbb Z}^{\ell}$).
There is a positive constant $c_{\ell}$ such that for every $n$,
$$\mu_n\left({\mathbb Z}^{\ell}\right)\geq c_{\ell}\,n^{3-\frac{1}{\ell}}.$$
\end{proposition}

\begin{proof}
With no loss of generality we may assume that $n>(6e)^{\ell}$. Let 
$$d_0:=\left\lfloor \ell\left(\frac{\sqrt[\ell]{n}}{2e}-1\right)\right\rfloor$$
and consider the ball $A=B_{{\mathbb Z}^{\ell}}(v_0; d_0)$ for an arbitrary $v_0\in{\mathbb Z}^{\ell}$.
Note that 
\begin{equation}\label{eq:ballZell}
|A|=\sum_{k=0}^{\ell}\binom{\ell}{k}2^k\binom{d_0}{k}\leq 2^{\ell}\sum_{k=0}^{\ell}\binom{\ell}{k}\binom{d_0}{k}=2^{\ell}\binom{d_0+\ell}{\ell}<\left(\frac{2e(d_0+\ell)}{\ell}\right)^{\ell}\leq n
\end{equation}
and for every positive integer $d$,
\begin{equation}\label{eq:sphereZell}
|S_{{\mathbb Z}^{\ell}}(v_0; d)|=\sum_{k=1}^{\ell}\binom{\ell}{k}2^k\binom{d-1}{k-1}\geq 2^{\ell}\binom{d-1}{\ell-1}.
\end{equation}
For $a=(a_1,a_2,\ldots,a_{\ell}),b=(b_1,b_2,\ldots,b_{\ell}),c=(c_1,c_2,\ldots,c_{\ell})\in {\mathbb Z}^{\ell}$, it holds that $(a,b,c)$ is a $3$-term arithmetic progression in ${\mathbb Z}^{\ell}$ if and only if $(b_i-a_i)(b_i-c_i)\leq 0$ for every $1\leq i\leq \ell$ and $\sum_{i=1}^{\ell}|a_i-b_i|=\sum_{i=1}^{\ell}|b_i-c_i|$.
In particular, if $b=(b_1,b_2,\ldots,b_{\ell})\in S_{{\mathbb Z}^{\ell}}(v_0; d)$ for some $0\leq d\leq d_0$,
then $(a,b,c)\in\AT {{\mathbb Z}^{\ell}} A$ 
for any $a=(a_1,a_2,\ldots,a_{\ell}),c=(c_1,c_2,\ldots,c_{\ell})\in {\mathbb Z}^{\ell}$ such that $a_i\geq b_i\geq c_i$ for every $1\leq i\leq \ell$ and 
$$\sum_{i=1}^{\ell}(a_i-b_i)=\sum_{i=1}^{\ell}(b_i-c_i)\leq d_0-d.$$
Hence, for every $0\leq d\leq d_0$ and every $b\in S_{{\mathbb Z}^{\ell}}(v_0;d)$,
$$w_A(b)\geq \sum_{k=0}^{d_0-d}\binom{k+\ell-1}{\ell-1}^2\geq\sum_{k=0}^{d_0-d}\binom{k+\ell-1}{2\ell-2}=\binom{d_0-d+\ell}{2\ell-1}.
$$
Therefore, by \eqref{eq:sum_w} and \eqref{eq:sphereZell},
\begin{align*}
|\AT {{\mathbb Z}^{\ell}} A|&=\sum_{d=0}^{d_0}\left(\sum_{b\in S_{{\mathbb Z}^{\ell}}(v_0; d)}w_A(b)\right)>\sum_{d=\ell}^{d_0-\ell+1}\left(\sum_{b\in S_{{\mathbb Z}^{\ell}}(v_0; d)}w_A(b)\right)\\
&\geq\sum_{d=\ell}^{d_0-\ell+1}2^{\ell}\binom{d-1}{\ell-1}\binom{d_0-d+\ell}{2\ell-1}=2^{\ell}\binom{d_0+\ell}{3\ell-1}.
\end{align*}
Hence, since $|A|<n$, by \eqref{eq:ballZell}, and $d_0>2\ell-1$,
\begin{align*}
\mu_n({\mathbb Z}^{\ell})&> \mu_{|A|}({\mathbb Z}^{\ell})\geq |\AT {{\mathbb Z}^{\ell}} A|> 2^{\ell}\binom{d_0+\ell}{3\ell-1}>2^{\ell}\left(\frac{d_0+\ell}{3\ell-1}\right)^{3\ell-1}\\
&>2^{\ell}\left(\frac{d_0+\ell+1}{3\ell}\right)^{3\ell-1}> 2^{\ell}\left(\frac{\sqrt[\ell]{n}}{6e}\right)^{3\ell-1}=\frac{2^{\ell}}{(6e)^{3\ell-1}}n^{3-\frac{1}{\ell}}.
\qedhere\end{align*}
\end{proof}

\begin{proposition}\label{prop:goodman}
For every metric space $M$ and every $n$,
\begin{equation}\label{eq:gen_upper}
\mu_n(M)\leq 2\left\lfloor \frac{n}{2}\left\lfloor\frac{n-1}{2}\right\rfloor\left\lceil \frac{n-1}{2}\right\rceil\right\rfloor+n=\frac{1}{4}n^3-\frac{1}{2}n^2+\begin{cases}
n & n \text{ is even},\\
\frac{5}{4}n & n\bmod 4=1,\\
\frac{5}{4}n-1 & n\bmod 4=3.
\end{cases}
\end{equation}
Moreover, there is a metric space $M_0$ such that for every $n$,
\begin{equation*}
\mu_n( M_0)=(n-2)\left\lfloor \frac{n}{2}\right\rfloor\left\lceil \frac{n}{2}\right\rceil+n=\frac{1}{4}n^3-\frac{1}{2}n^2+\begin{cases}
n & n \text{ is even},\\
\frac{3}{4}n+\frac{1}{2} & n \text{ is odd}.
\end{cases}
\end{equation*}
\end{proposition}
\begin{proof}
Let $A$ be a set of $n$ points in some metric space $M$. For every $a\in A$ we may partition the set $A\setminus\{a\}$ to two sets $A_1,A_2$ such that none of them contains points $b,c$ for which $d_M(a,c)=2\,d_M(a,b)$ (for instance, we may take $A_1$ to be the set of $x\in A\setminus\{a\}$ for which $\lfloor\log_2 d_M(a,x)\rfloor$ is odd and $A_2$ to be the set of $x\in A\setminus\{a\}$ for which $\lfloor\log_2 d_M(a,x)\rfloor$ is even).
Hence,
$$|\{(x,y,z)\in\AT M A: x=a\}|\leq 1+|A_1|\,|A_2|\leq 1+\left\lfloor \frac{n-1}{2}\right\rfloor \left\lceil \frac{n-1}{2}\right\rceil.$$
It follows that
$$|\AT M A|=\sum_{a\in A}|\{(x,y,z)\in\AT M A: x=a\}|\leq n+n\left\lfloor \frac{n-1}{2}\right\rfloor \left\lceil \frac{n-1}{2}\right\rceil,$$
and hence, by Observation \ref{obs:parity},
$$|\AT M A|\leq n+2\left\lfloor\frac{n}{2}\left\lfloor \frac{n-1}{2}\right\rfloor \left\lceil \frac{n-1}{2}\right\rceil\right\rfloor,$$
which proves \eqref{eq:gen_upper}.

To get the second part, consider a complete bipartite graph with infinite parts $L$ and $R$, with respect to the graph metric, i.e., $d_{L\cup R}(x,y)=2$ if $x,y$ are distinct vertices in the same part and $d_{L\cup R}(x,y)=1$ if $x,y$ are in different parts. For every finite $A\subset L\cup R$, clearly
\begin{align*}|\AT {L\cup R} A|&=|A\cap L|\, |A\cap R|(|A\cap L|-1)+|A\cap R|\, |A\cap L|(|A\cap R|-1)+|A|\\
&=(|A|-2)|A\cap R|\, |A\cap L|+|A|.\end{align*}
Hence, for every $n$,
\begin{equation*}
\mu_n(L\cup R)=(n-2)\left\lfloor n/2\right\rfloor\left\lceil n/2\right\rceil+n.
\qedhere\end{equation*}
\end{proof}

\begin{rem}
The argument proving \eqref{eq:gen_upper} may also be presented as follows. Let $A$ be a set of $n$ points in some metric space $M$. We may partition the finite set $\{d_M(a,b)\mid a,b\in A,\, a\neq b\}$ of positive real numbers to two sets $D_1,D_2$ such that none of them contains a number $\alpha$ and the number $2\alpha$. Now consider the graph on the vertex set $A$ such that distinct vertices $a,b\in A$ are adjacent if $d_M(a,b)\in D_1$. 
For any $0\leq i\leq 3$, let $r_i$ denote the number of sets of three distinct vertices of the graph having exactly $i$ edges between them. Clearly,
$$|\AT M A|\leq n+2(r_1+r_2)=n+2\binom{n}{3}-2(r_0+r_3),$$
and the result follows, since
$$r_0+r_3\geq \frac{1}{24}n^3-\frac{1}{4}n^2+\begin{cases}
\frac{1}{3}n & n \text{ is even},\\
\frac{5}{24}n & n\bmod 4=1,\\
\frac{5}{24}n+\frac{1}{2} & n\bmod 4=3,
\end{cases}$$
by Goodman's theorem \cite{goodman}.
\end{rem}

\section{Concluding remarks and open questions}\label{sec:open}

In this paper we studied the maximal number of $3$-term arithmetic progressions in $n$-element subsets of a metric space. 
We suggest a few directions for future research.

\subsection*{The extent to which the result of Green and Sisask extends}
Green and Sisask proved that the maximal number of $3$-term arithmetic progressions in $n$-element sets of integers is $\lceil n^2/2\rceil$. In Theorem \ref{thm:hyper} we showed that the same holds if the set of integers is replaced by any uniquely geodesic Riemannian manifold in which every local geodesic is a geodesic. In particular, the result of Green and Sisask is valid in Cartan--Hadamard manifolds, i.e., complete simply connected Riemannian manifolds that have everywhere nonpositive sectional curvature. 
However, \eqref{eq:trees} implies that this result does not extend to the wider class of Hadamard spaces, i.e., complete globally nonpositively curved (in the sense of A. D. Alexandrov) metric spaces.

It would be interesting to understand what is the largest natural family of metric spaces to which the result of Greem and Sisask extends.

\subsection*{Spherical geometry}
In Theorem \ref{thm:S1} we determined the maximal number of $3$-term arithmetic progressions in $n$-element subsets of the $1$-dimensional sphere $S^1$.

It would be very interesting to confirm or refute Conjecture \ref{conj:S2} regarding the $2$-dimensional sphere $S^2$ and proceed to understand the maximal number of $3$-term arithmetic progressions in $n$-element subsets of higher dimensional spheres. 

\subsection*{Asymptotic bounds}
It would be interesting to determine, at least asymptotically, the maximal number of $3$-term arithmetic progressions in $n$-element subsets of additional metric spaces.
In particular, it would be interesting to decide whether the asymptotic lower bounds \eqref{eq:trees}, in the case of the $r$-regular tree ${\mathbb T}_r$, and \eqref{eq:lattices}, in the case of the lattice graph ${\mathbb Z}^{\ell}$, are tight.

\subsection*{Equilateral (and other) triangles}
A $3$-term arithmetic progression in a metric space may also be viewed as a (degenerate) isosceles triangle. One may ask what is the maximal number of triangles of other types in $n$-element subsets of a metric space. In the Euclidean setting, this question was already studied in \cite{EP} for various types of triangles, including equilateral triangles. For a metric space $M$ and a positive integer $n$, let us denote:
$$\eta_n(M):=\max_{A\subseteq M,\, |A|=n}\left|\left\{(a,b,c)\in A^3: d_M(a,b)=d_M(a,c)=d_M(b,c)>0\right\}\right|.$$
It was shown in \cite{EP} that $\eta_n({\mathbb R}^2)=\Theta(n^2)$,
$\eta_n({\mathbb R}^3)=O(n^{7/3})$ (later improved to $\eta_n({\mathbb R}^3)=O(n^{11/5})$ in \cite{ATT}), 
$\eta_n({\mathbb R}^4)=O(n^{8/3})$, 
$\eta_n({\mathbb R}^5)=O(n^{26/9})$ 
and $\eta_n({\mathbb R}^k)=\Theta(n^3)$ for any $k\geq 6$.
It would be interesting to understand the asymptotics of $\eta_n$ in other metric spaces. 
It is simple to show that $\eta_n({\mathbb T}_r)=\Theta(n^3)$ for any $r\geq 3$ (for the lower bound, take points on a sphere around an arbitrary vertex $v_0$, evenly distributed among the connected components of the graph that is obtained from the tree by removing $v_0$).
The argument in \cite{EP} showing the quadratic upper bound in the Euclidean plane is valid for the hyperbolic plane as well, but we believe that in fact $\eta_n({\mathbb H}^2)=\Theta(n)$.

\subsection*{Longer arithmetic progressions}
The result of Green and Sisask on the maximal number of $3$-term arithmetic progressions in $n$-element sets of integers was recently extended to arithmetic progressions of any length in \cite[\S 8.1]{BGSZ}, where it was proved that for any $k$, the maximal number of $k$-term arithmetic progressions in $n$-element sets of integers is attained for $n$-term arithmetic progressions (among other sets).
Another generalization of the problem for integers was studied in \cite{A}.

It would be interesting to study the maximal number of $k$-term arithmetic progressions in $n$-element subsets of a metric space for $k>3$.


\begin{thebibliography}{999999999}

\bibitem{A}
J. Aaronson, Maximising the number of solutions to a linear equation in a set of integers, Bull. Lond. Math. Soc. {\bf 51} (2019), no.~4, 577--594.

\bibitem{ATT}
T. Akutsu, H. Tamaki\ and\ T. Tokuyama, Distribution of distances and triangles in a point set and algorithms for computing the largest common point sets, Discrete Comput. Geom. {\bf 20} (1998), no.~3, 307--331.

\bibitem{BGS}
W. Ballmann, M. Gromov\ and\ V. Schroeder, {\it Manifolds of nonpositive curvature}, Progress in Mathematics, 61, Birkh\"{a}user Boston, Inc., Boston, MA, 1985.

\bibitem{BGSZ}
B. B. Bhattacharya, S. Ganguly, X. Shao, and Y. Zhao, Upper tail large deviations for arithmetic progressions in a random set, Int. Math. Res. Not. IMRN {\bf 2020}, no.~1, 167--213. 

\bibitem{BH}
M. R. Bridson\ and\ A. Haefliger, {\it Metric spaces of non-positive curvature}, Grundlehren der Mathematischen Wissenschaften, 319, Springer-Verlag, Berlin, 1999.

\bibitem{dBE}
N. G. de Bruijn\ and\ P. Erd\H{o}s, On a combinatorial problem, Nederl. Akad. Wetensch., Proc. {\bf 51} (1948), 1277--1279 = Indagationes Math. {\bf 10}, 421--423 (1948).

\bibitem{BBI}
D. Burago, Y. Burago\ and\ S. Ivanov, {\it A course in metric geometry}, Graduate Studies in Mathematics, 33, American Mathematical Society, Providence, RI, 2001.

\bibitem{BZ}
Yu. D. Burago\ and\ V. A. Zalgaller, {\it Geometric inequalities}, translated from the Russian by A. B. Sosinski\u{\i}, Grundlehren der Mathematischen Wissenschaften, 285, Springer-Verlag, Berlin, 1988. 

\bibitem{Cha}
I. Chavel, {\it Isoperimetric inequalities: differential geometric and analytic perspectives}, Cambridge Tracts in Mathematics, 145, Cambridge University Press, Cambridge, 2001.

\bibitem{CE}
J. Cheeger\ and\ D. G. Ebin, {\it Comparison theorems in Riemannian geometry}, North-Holland Publishing Co., Amsterdam, 1975.

\bibitem{EP}
P. Erd\H{o}s\ and\ G. Purdy, Some extremal problems in geometry. III, in {\it Proceedings of the Sixth Southeastern Conference on Combinatorics, Graph Theory and Computing Florida Atlantic Univ., Boca Raton, Fla., 1975)}, 291--308. Congressus Numerantium, XIV, Utilitas Math., Winnipeg, MB.

\bibitem{goodman}
A. W. Goodman, On sets of acquaintances and strangers at any party, Am. Math Mon 66 (1959), 778--783.

\bibitem{GS} 
B. Green\ and\ O. Sisask, On the maximal number of 3-term arithmetic progressions in subsets of $\mathbb Z/p\mathbb Z$, 
Bull. Lond. Math. Soc. {\bf 40} (2008), no.~6, 945--955.

\bibitem{Gro}
M. Gromov, Paul Levy’s isoperimetric inequality, Pr\'{e}publications I.H.E.S. (1980).

\bibitem{KK}
B. R. Kloeckner\ and\ G. Kuperberg, The Cartan-Hadamard conjecture and the Little Prince, Rev. Mat. Iberoam. {\bf 35} (2019), no.~4, 1195--1258.

\bibitem{Lea}
I. Leader, Discrete isoperimetric inequalities, in {\it Probabilistic combinatorics and its applications (San Francisco, CA, 1991)}, 57--80, Proc. Sympos. Appl. Math., 44, AMS Short Course Lecture Notes, Amer. Math. Soc., Providence, RI.

\bibitem{Oss}
R. Osserman, The isoperimetric inequality, Bull. Amer. Math. Soc. {\bf 84} (1978), no.~6, 1182--1238.

\end{thebibliography}
\end{document}